\documentclass[11pt]{amsart}
\usepackage[all]{xy}

\begin{document}

\title{Cohomologies of finite abelian groups}
\author[Yu. Drozd, A. Plakosh]{Yuriy A. Drozd and Andriana I. Plakosh}
\address{Institute of Mathematics, National Academy of Sciences of Ukraine, 
 Tereschenkivska 3, 01601 Kyiv, Ukraine}
 \email{y.a.drozd@gmail.com, drozd@imath.kiev.ua, 
 andrianaplakoshmail@gmail.com}
 \urladdr{www.imath.kiev.ua/$\sim$drozd}
 


\newtheorem{theorem}{Theorem}[section]
\newtheorem{prop}[theorem]{Proposition}
\newtheorem{corol}[theorem]{Corollary}
\numberwithin{equation}{section}

\def\tM{\tilde{M}}				\def\hH{\hat{H}}				
\def\mZ{\mathbb{Z}}		\def\mQ{\mathbb{Q}}
\def\mS{\mathbb{S}}		\def\mF{\mathbb{F}}
\def\mP{\mathbb{P}}		\def\mT{\mathbb{T}}
\def\qR{{\boldsymbol R}}		\def\qK{{\boldsymbol K}}
\def\qP{{\boldsymbol P}}

\def\si{\sigma}			\def\ga{\gamma}
\def\ze{\zeta}			\def\la{\lambda}
\def\eps{\varepsilon}	\def\de{\delta}

\def\*{\otimes}			\def\bop{\bigoplus}
\def\xx{\times}			\def\dd{\partial}
\def\8{\infty}				\def\Arr{\Rightarrow}
\def\+{\oplus}			\def\xarr{\xrightarrow}

\def\setsuch#1#2{\left\{\,#1\mid #2\,\right\}}
\def\gnrsuch#1#2{\langle\,#1\mid #2\,\rangle}
\def\gnr#1{\langle\,#1\,\rangle}
\def\lst#1#2{ #1_1 , #1_2 , \dots , #1_{#2} }
\def\bop{\bigoplus}

\def\tor{\mathop\mathrm{Tor}\nolimits}
\def\ord{\mathop\mathrm{ord}\nolimits}
\def\cok{\mathop\mathrm{coker}\nolimits}
\def\tr{\mathop\mathrm{tr}\nolimits}
\def\Hom{\mathop\mathrm{Hom}\nolimits}
\def\id{\mathrm{id}}
\def\res{\mathop\mathrm{res}\nolimits}
\def\ss{\mathrm{s}}

\def\iff{if and only if }
\def\hhh{\phantom{hhh}}

\maketitle

\begin{abstract}
  We construct a simplified resolution for the trivial $G$-module $\mZ$, where $G$ is a finite abelian group, and compare it with
  the standard resolution. We use it to calculate cohomologies of irreducible $G$-lattices and their duals.
\end{abstract}

\section*{Introduction}

 The theory of cohomologies of groups was inspired by the works of Hurewicz on cohomologies of acyclic spaces and was founded
 in 1940's by Eilenberg--MacLane, Eckmann, Hopf and others. It was one of the origins of the homological algebra. It was
 also related to the theory of group extensions and projective representations, where cohomologies arise as factor sets. 
 This theory is widely used in topology, number theory, algebraic geometry and other branches of mathematics. Thus it
 is actively studied by plenty of mathematicians. In particular, there is a lot of papers devoted to the calculation of cohomologies 
 of concrete groups and their classes. In these investigations one often needs special sorts of resolutions, which are simpler and
 more convenient than the standard one. For instance, Takahashi \cite{tak} proposed a new approach to the calculation of
 cohomologies of finite abelian groups and gave applications of his method to the cohomologies of the trivial module and 
 of some Galois groups.
 
 The aim of our paper is to describe a rather simple resolution for finite abelian groups (Section~\ref{s1}) and to use it for 
 calculation of cohomologies of irreducible $G$-lattices and their duals (Sections~\ref{s4} and \ref{s5}). Our approach is close to 
 that of Takahashi, though it seems more explicit. We also compare our resolution with the standard one (Section~\ref{s2}) 
 and prove some facts concerning duality for cohomologies of $G$-lattices (Section~\ref{s3}). The results about the second
 cohomologies can be useful in the study of crystallographic groups and of Chernikov groups.

\section{Resolution}
\label{s1} 

 For a periodic element $a$ of a group $G$ we denote by $o(a)$ the order of $a$, $s_a=\sum_{i=0}^{o(a)-1}a^i$.
 Let $G=\prod_{i=1}^sG_i$ be a direct product of finite cyclic groups $G_i=\gnrsuch{a_i}{a_i^{o_i}=1}$ of orders $o_i=o(a_i)$, 
 $\qR=\mZ G$, $\mP=\qR[\lst xs]$
 and $\mP_n$ be the set of homogeneous polynomials from $\mP$ of degree $n$ (including $0$). We define a differential 
 $d:\mP_n\to\mP_{n-1}$ by the rule
\[
 d_n(x_1^{k_1}x_2^{k_2}\dots x_s^{k_s})=\sum_{i=1}^s(-1)^{K_i} C_ix_1^{k_1}\dots x_i^{k_i-1}\hspace*{-3pt}\dots x_s^{k_s},
\]
where $K_i=\sum_{j=1}^{i-1}k_j$ and
\[
 C_i=
 \begin{cases}
  a_i-1 &\text{if $k_i$ is odd},\\
  s_{a_i} &\text{if $k_i>0$ is even},\\
  0 &\text{if } k_i=0.
 \end{cases}
\]24 (2017), 144–157

 When speaking of the $G$-module $\mZ$, we always suppose that the elements of $G$ act trivially.

\begin{theorem}\label{resolution} 
 $\mP=(\mP_n,d_n)$ is a free resolution of the $G$-module $\mZ$.
\end{theorem}
\begin{proof}
  If $s=1$, it is well-known. If $\qR_i=\mZ G_i$ and $\mP^i$ denotes such resolution for the group $G_i$, then $\qR=\bigotimes_{i=1}^s\qR_i$
 and $\mP$ is the tensor product of complexes $\bigotimes_{i=1}^s\mP^i$. As all groups of cycles and boundaries in the complexes
 $\mP^i$ are free abelian, the claim follows from the K\"unneth relations \cite[Theorem~VI.3.1]{ce}.
\end{proof}

 \section{Correspondence with standard resolution}
 \label{s2} 
 
 To apply Theorem~\ref{resolution}, for instance, to extensions of groups, we have to compare it with the standard resolution, which is
 usually used for this purpose \cite{br,ce}. So, in what follows, $\mS$ denotes the normalized standard resolution for $\mZ$ as 
 $\qR$-module, $\setsuch{[\lst gn]}{g_i\in G\setminus\{1\}}$ is the usual basis of $\mS_n$ such that the standard differential $d^\ss$ is
 defined as
 \begin{align*}
  d^\ss_n[\lst gn]=g_1[g_2, \dots, g_n]&+\sum_{i=1}^n(-1)^i[g_1,\dots, g_ig_{i+1},\dots, g_n]+\\& +(-1)^n[\lst g{n-1}],  
 \end{align*}
 setting $[\lst gn]=0$ if some $g_i=1$. Note that $\mP_0=\mS_0=\qR$. 
 
 We denote $a^{\{i\}}=1+a+a^2+\dots a^{i-1}$. Then $s_a=a^{\{o(a)\}}$,
  \begin{align}\label{eq0} 
 & a^{\{i+k\}}= a^{\{i\}}+a^ia^{\{k\}},\\
  \intertext{in particular,}
 & a^{\{m+o(a)\}}=a^{\{m\}}+a^ms_a. \notag
  \end{align} 
  
 \begin{theorem}
 There is a quasi-isomorphism $\si:\mS\to\mP$ such that 
 \begin{equation}\label{eq2} 
 \begin{split}
   &\hspace*{-.5em} \si_0=\id, \\
 &\hspace*{-.5em} \si_1[a_1^{k_1}a_2^{k_2}\dots a_s^{k_s}]=\sum_{i=1}^s \big(\prod_{j=1}^{i-1} a_j^{k_j}\big) a_i^{\{k_i\}}x_i, \\
 &\hspace*{-.5em} \si_2 [a_1^{k_1}a_2^{k_2}\dots a_s^{k_s}\!,a_1^{l_1}a_2^{l_2}\dots a_s^{l_s}]{=}
  \sum_{i=1}^s\sum_{j=1}^i \Big(\prod_{q=1}^{i-1}a_q^{k_q} \prod_{r=1}^{j-1} a_r^{l_r}\Big)\, \si_2[a_i^{k_i},a_j^{l_j}],\\
  &\hspace*{-.5em}\text{\emph{where}\ }\
   \si_2[a_i^k,a_j^l]=
  \begin{cases}
   [(k+l)/o_i]x_i^2 &\text{\emph{if} }  i=j,\\
   0 &\text{\emph{if} }  i<j,\\
   a_j^{\{l\}}a_i^{\{k\}}x_jx_i &\text{\emph{if} }  i>j
  \end{cases}
 \end{split} 
 \end{equation}24 (2017), 144–157
 \end{theorem}
 
 Since $\mS$ and $\mP$ are free resolutions of $\mZ$, $\si$ induces isomoprhisms of cohomologies
 $H^n(\Hom_\qR(\mS,M))\to H^n(\Hom_\qR(\mP,M))$. In particular, combining $\si_2$ with
 cocycles from $\Hom_\qR(\mP_2,M)$, we obtain the ``usual'' presentation of cocycles from $H^2(G,M)$.
 
 \begin{proof}
  Actually, we have to show that the diagram
    \[
  \xymatrix{ \mS_2 \ar[r]^{d^\ss_2} \ar[d]_{\si_2} & \mS_1 \ar[d]_{\si_1} \ar[r]^{d^\ss_1} &\mS_0 \ar@{=}[d] \\
  				 \mP_2 \ar[r]^{d_2} & \mP_1 \ar[r]^{d_1} & \mP_0.
   }
 \]
 is commutative. Then the set of homomorphisms $\{\si_0,\si_1,\si_2\}$ extends to a quasi-isomorphism $\si:\mS\to\mP$.
 
 Note that $gh-1=(g-1)+g(h-1)$ and $a^k-1=a^{\{k\}}(a-1)$. Therefore,
 \begin{align*}
& d_1^\ss[a_1^{k_1}a_2^{k_2}\dots a_s^{k_s}]=a_1^{k_1}a_2^{k_2}\dots a_s^{k_s}-1
 		=\sum_{i=1}^s\big(\prod_{j=1}^{i-1}a_j^{k_j}\big)(a_i^{k_i}-1)=\\
 		&\ =\sum_{i=1}^s\big(\prod_{j=1}^{i-1}a_j^{k_j}\big)a_i^{\{k_i\}}(a_i-1)=
 		\sum_{i=1}^s\big(\prod_{j=1}^{i-1} a_j^{k_j}\big) a_i^{\{k_i\}}d_1x_i,
 \end{align*}
 hence $d_1\si_1=d_1^\ss$.
 
  Set $(r)_i=\res(r,o_i)$, the residue of $r$ modulo $o_i$. Then, for $0\le k<o_i$, $0\le l<o_i$,
  \begin{align*}
  & d_2^s[a_i^k,a_i^l]=\,a_i^k[a_i^l]-[a_i^{k+l}]+[a_i^k], \\ \intertext{thus}
  &\si_1d_2^\ss[a_i^k,a_i^l]=(a_i^ka_i^{\{l\}}-a_i^{\{(k+l)_i\}}+a_i^{\{l\}})x_i=\\
  &\qquad =(a_i^ka_i^{\{l\}}-a_i^{\{k+l\}}+a_i^{\{l\}}+[(k+l)/o_i]s_{a_i})x_i=\\
  & \qquad    =[(k+l)/o_i]s_{a_i}x_i=\\
  &\qquad =d_2([(k+l)/o_i]x^2_i),\\
  \intertext{so, if we set}
  & \si_2[a_i^k,a_i^l]=[(k+l)/o_i]x^2_i,\\
   \intertext{we have}
   &d_2\si_2[a_i^k,\,a_i^l]=\si_1d^\ss_2[a_i^k,a_i^l].\\
   \intertext{In the same way,} 
  &d_2^\ss[a_i^k,a_j^l]=\, a_i^k[a_j^l]-[a_i^ka_j^l]+[a_i^k],\\ 
  \intertext{thus, if $i<j$,}
  &\si_1 d_2^\ss[a_i^k,a_j^l]= a_i^ka_j^{\{l\}}x_j-a_i^{\{k\}}x_i-a_i^ka_j^{\{l\}}x_j+a_i^{\{k\}}x_i=0,\\
  \intertext{while if $i>j$}
  & \si_1d_2^\ss[a_i^k,a_j^l]=a_i^ka_j^{\{l\}}x_j-a_j^{\{l\}}x_j-a_j^la_i^{\{k\}}x_i+a_i^{\{k\}}x_i=\\
  &\qquad	= (a_i^k-1)a_j^{\{l\}}x_j - (a_j^l-1)a_i^{\{k\}}x_i =\\
  &\qquad	= -d_2(a_j^{\{l\}}a_i^{\{k\}}x_jx_i).
  \intertext{So, if we set}
  &	 \si_2[a_i^k,a_j^l]=
  \begin{cases}
   0 &\text{if }  i<j,\\
   a_j^{\{l\}}a_i^{\{k\}}x_jx_i &\text{if }  i>j
  \end{cases}
\intertext{we have}24 (2017), 144–157
  &d_2\si_2[a_i^k,\,a_j^l]=\si_1d_2^\ss [a_i^k,a_j^l]
    \end{align*}
    for $i\ne j$.
    
   Let now $\si_2$ is defined by the rule \eqref{eq2}. We check that $d_2\si_2=\si_1d_2^\ss$ for $s=3$. The general case
   is analogous, though a bit cumbersome. We write $a,b,c$ instead of $a_1,a_2,a_3$ and $x,y,z$ instead of $x_1,x_2,x_3$.
 Then
 \begin{align*}
   & \si_1d_2^\ss[a^ib^jc^r,a^kb^lc^s]= \\
   &\ =\si_1(a^ib^jc^r[a^kb^lc^s]-[a^{i+k}b^{j+l}c^{r+s}]+[a^ib^jc^r]) = \\
   &\ = a^ib^jc^r(a^{\{k\}}x+a^kb^{\{l\}}y+a^kb^lc^{\{s\}}z)-\\
   &\ -a^{\{i+k\}}x-a^{i+k}b^{\{j+l\}}y-a^{i+k}b^{j+l}c^{\{r+s\}}z +\\
   &\ +[(i+k)/o_a]s_a x +a^{i+k}[(j+l)/o_b]s_by+  a^{i+k}b^{j+l}[(r+s)/o_c]s_c x +\\
  &\ + a^{\{i\}}x +a^ib^{\{j\}}y +a^ib^jc^{\{r\}}z=\\
  &\ =(a^ib^jc^ra^{\{k\}}-a^{\{i+k\}}+a^{\{i\}}+[(i+k)/o_a]s_a)x+\\
  &\ +a^i(a^kb^jc^rb^{\{l\}}-a^kb^{\{j+l\}}+b^{\{j\}}+a^k[(j+l)/o_b]s_b)y+\\
  &\ +a^ib^j(a^kb^lc^rc^{\{s\}}-a^kb^lc^{\{r+s\}}+c^{\{r\}}+a^kb^l[(r+s)/o_c]s_c)z,
 \intertext{while} 
  &d_2\si_2[a^ib^jc^r,a^kb^lc^s]= \\24 (2017), 144–157
 &\ =d_2 \big(\!-a^ia^{\{k\}}b^{\{j\}}xy - a^ib^ja^{\{k\}}c^{\{s\}}xz - a^{i+k}b^jb^{\{l\}}c^{\{r\}}yz +\\
 &\ +[(i+k)/o_a]x^2 +a^{i+k}[(j+l)/o_b]y^2+  a^{i+k}b^{j+l}[(r+s)/o_c]z^2\big) =\\
 & \ =-a^i(a^k-1)b^{\{j\}}y+a^i(b^j-1)a^{\{k\}}x -a^i(a^k-1)b^jc^{\{r\}}z +\\
 &\ + a^ib^j(c^r-1)a^{\{k\}}x -a^{i+k}b^j(b^l-1)c^{\{r\}}z +a^{i+k}b^j(c^r-1)b^{k\}}y,\\
  &\ +[(i+k)/o_a]s_a x +a^{i+k}[(j+l)/o_b]s_by+  a^{i+k}b^{j+l}[(r+s)/o_c]s_c x=\\
  &\ (-a^ia^{\{k\}} +a^ib^jc^ra^{\{k\}}+[(i+k)/o_a]s_a )x +\\
  &\ +a^i(-a^kb^{\{j\}}+b^{\{j\}}+a^kb^jc^rb^{\{l\}}-a^kb^jb^{\{l\}}+a^k[(j+l)/o_b]s_b)y+\\
  &\ +a^ib^j(c^{\{r\}}-a^kb^lc^{\{r\}}+a^kb^l[(r+s)/o_c]s_c)z.
 \end{align*}
 Relations \eqref{eq0} immediately imply that both results are equal.
 \end{proof}
 
 \section{Cohomologies of $G$-lattices.}
 \label{s3}
 
 In this section $G$ denotes a finite group, $\qR=\mZ G$.
 Recall that a \emph{$G$-lattice} (or an \emph{integral representation} of $G$) is a $G$-module $M$ such that its abelian group
 is free of finite rank. They also say that $M$ is \emph{a lattice in the $\mQ G$-module $\tM=\mQ\*_\mZ M$}. 
 Two $G$-lattices $M,N$ are said to be \emph{of the same genus} if $M_p\simeq N_p$24 (2017), 144–157
 for each prime $p$, where $M_p=\mZ_p\*_\mZ M$ ($\mZ_p=\setsuch{r/z}{r\in\mZ,\,s\in\mZ\setminus{p\mZ}}$).
Then they write $M\vee N$. We also set $M^*=\Hom_\mZ(M,\mZ)$, where $G$ acts by the rule $gf(u)=f(g^{-1}u)$.
 
 We denote by $\hH^n(G,M)$ the \emph{Tate cohomologies} of $G$ with coefficients in $M$ \cite{br,ce}. Let
 \[
  \mF: \dots \to\mF_n\xarr{d_n} \mF_{n-1}\xarr{d_{n-1}} \dots\xarr{d_2} \mF_1\xarr{d_1} \mF_0\to0
 \]
 be a free resolution of $\mZ$, where all modules $\mF_n$ are finitely generated,
 \[
  \mF^*: 0\to \mF_0^* \xarr{d^*_1} \mF_1^*\xarr{d^*_2} \dots\xarr{d^*_{n-1}} \mF_{n-1}^*\xarr{d^*_n}\mF_n^*\to \dots
 \]
 be the dual complex, $d_0:\mF_0\to\mF^*_0$ be the composition of the maps 
 $\mF_0\to\cok d_1\simeq\mZ\simeq\ker d^*_0\to\mF^*_0$. Set $\mF_{-n}=\mF^*_{n-1},\,d_{-n}=d^*_n$. The sequence
 \begin{align*}
 \mF^+: \dots& \to\mF_n\xarr{d_n} \mF_{n-1}\xarr{d_{n-1}} \dots\xarr{d_2} \mF_1\xarr{d_1} \mF_0\xarr{d_0}\\
 & \xarr{d_0}\mF_{-1} \xarr{d_{-1}} \mF_{-2}\xarr{d_{-2}} \dots\xarr{d_{-n}} \mF_{-n}^*\xarr{d_{-n}}\mF_{-n-1}\to \dots
 \end{align*}
 is called a \emph{complete resolution} for the group $G$. Then $\hH^n(G,M)$ are just the cohomologies of the complex
 $\Hom_{\qR}(\mF^+,M)$. If $\mF_0=\qR$ and the surjection $\mF_0\to\mZ$ maps $g$ to $1$, then $\mF_{-1}\simeq\qR$
 and $d_0$ is just the \emph{trace}, i.e. the multiplication by $\tr_G=\sum_{x\in G}x$. It is the case for the resolutions $\mF$
 and $\mS$. 

 \begin{prop}\label{31} 
  Let $G$ be a finite group, $M,N$ be $G$-lattices such that $M\vee N$. Then $\hH^n(G,M)\simeq \hH^n(G,N)$ 
 for all $n$.
 \end{prop}
 \begin{proof}
  It is known that all groups $\hH^n(G,M)\ (n>0)$ are periodic of period $\#(G)$, hence 
   $\hH^n(G,M)\simeq \bop_{p|\#(G)} \hH^n(G,M)_p$. Moreover, as $\mZ_p$ is flat over $\mZ$, 
  $\hH^n(G,M)_p\simeq \hH^n(G,M_p)$. It implies he claim.
 \end{proof}
 
 We denote by $DM$ the dual $G$-module $DM=\Hom_\mZ(M,\mT)$, where $\mT=\mQ/\mZ$. 
 
 \begin{prop}\label{32} 
  Let $M$ be a $G$-lattice. Then
  \begin{align}
  &  \hH^{n-1}(G,DM)\simeq D\hH^{-n}(G,M), \label{e31} \\  
   &\hH^n(G,DM)\simeq \hH^{n+1}(G,M^*),		\label{e32}\\
   & \hH^n(G,M^*)\simeq D\hH^{-n}(G,M). \label{e33} 
  \end{align}
 \end{prop}
 
 If $M=\mZ$, \eqref{e33} coincides with \cite[Theorem~XII.6.6]{ce}.
 
 \begin{proof}
  \eqref{31} follows from \cite[Corollary~XII.6.5]{ce}.
  
  Consider the exact sequence $0\to\mZ\to\mQ\to\mT\to0$. As $M$ is free abelian, it gives the exact sequence of $G$-modules
  \[
   0\to M^* \to \Hom_\mZ(M,\mQ) \to DM \to 0.
  \]
  $\hH^n(G,\Hom_\mZ(M,\mQ))=0$ for all $n$, since the multiplication by $\#(G)$ is an automorphism of $\Hom_\mZ(M,\mQ)$,  
  whence we obtain \eqref{e32}. 
  
  \eqref{e33} follows from \eqref{e31} and \eqref{e32}.
 \end{proof}
 
 We also need some information on cohomologies of direct products.
 
 \begin{prop}\label{33} 
  Let $N$ be a normal subgroup of $G$, $F=G/N$ and $\gcd(\#(N),\#(F))=1$. For every $G$-module $M$ and all $n$
  \begin{equation}\label{e34} 
     \hH^n(G,M)\simeq \hH^n(N,M)^F\+\hH^n(F,M^N).
  \end{equation}
 \end{prop}
 \begin{proof}
  As $\#(G)$ annihilates all $H^n(G,M)$ if $n>0$ and the same is true for $N$ and $F$, in the Hochschild--Serre spectral
  sequence
  \[
   H^p(F,H^q(N,M)) \Longrightarrow H^n(G,M)
  \]
  all terms with $p>0$ and $q>0$ are zero. Hence, if $n>0$,
  \begin{align*}
   \hH^n(G,M)&\simeq H^0(F,\hH^n(N,M))\+\hH^n(F,H^0(N,M))=\\ =&\hH^n(N,M)^F\+\hH^n(F,M^N).  
  \end{align*}
  Suppose now that the claim holds for $\hH^n$. Choose an exact sequence $0\to L\to P\to M\to0$, where $P$ is a
  free $\mZ G$-module. Then 
  \[
   \hH^{n-1}(G,M)\simeq \hH^n(G,L)\simeq \hH^n(N,L)^F\+\hH^n(F,L^N).
  \]
  As $P$ is also free as $\mZ N$-module, $\hH^n(N,L)\simeq \hH^{n-1}(N,M)$. On the other hand, there are exact sequences
  \[
   0\to L^N\to P^N \to M'\to 0 
   \]
  and
  \[
  0\to M'\to M^N\to M^N/M'\to 0,
  \]
  where $M'$ is the image of the map $P^N\to M^N$. Obviously, $M'\supseteq \tr_NM$, thus $\#(N)(M^N/M')=0$,
  whence $\hH^n(F,M^N/M')=0$. Therefore,
  \[
   \hH^{n-1}(F,M^N)\simeq \hH^{n-1}(F,M')\simeq \hH^n(F,L^N),
  \]
  since $P^N$ is a free $\mZ F$-module. So the isomorphism \eqref{34} holds for $\hH^{n-1}$, hence for all values of $n$.
  \end{proof}
  
 \begin{corol}\label{34} 
  Let $G=G_1\xx G_2$ with $\gcd(\#(G_1),\#(G_2))=1$, $M=M_1\*_\mZ M_2$, where $M_i$ is a $G_i$-lattice \textup{(}$i=1,2$\textup{)}. 
  Then
  \[
   \hH^n(G,M)\simeq \hH^n(G_1,M_1)\*_\mZ M_2^{G_2}\+ M_1^{G_1}\*_\mZ \hH^n(G_2,M_2).
  \]
 \end{corol}
 \begin{proof}
  As $M_i$ are free abelian, $\*_\mZ M_i$ is an exact functor and $M^{G_i}=M_i^{G_i}\*_\mZ M_j\ (j\ne i)$. Hence 
  $\hH^n(G_i,M)\simeq \hH^n(G_i,M_i)\*_\mZ M_j$, where $j\ne i$. So the claim is just a reformulation of 
  Proposition~\ref{33} for this special case.
 \end{proof}
    
 \section{Cohomologies of irreducible $G$-lattices}
 \label{s4}
 
 A $G$-lattice $M$ is called \emph{irreducible} if there are no submodules $0\ne N\subset M$ such that $M/N$ 
 is torsion free (i.e. again a $G$-lattice). Equivalently, $\tM=\mQ\*_\mZ M$ is a simple $\mQ G$-module. 
 If $G$ is a finite abelian group, then any simple $\mQ G$-module is defined by a group homomorphism $\rho:G\to\qK^\xx$,
 where $\qK$ is a cyclotomic field and the image of $\rho$ generates the ring of integers of $\qK$. Therefore, any two
 $G$-lattices in $\qK$ are of the same genus \cite{cr}, so have the same cohomologies. In particular, if $M$ is a $G$-lattice
 in $\qK$, so is $M^*$, hence $M^*\vee M$ and 
 \begin{equation}\label{e41} 
 \hH^n(G,M)\simeq \hH^n(G,M^*)\simeq D\hH^{-n}(G,M)\simeq D\hH^{n-1}(G,DM).
 \end{equation}
 The subgroup of periodic elements of $\qK$ is cyclic and generated by a primitive root of unity $\ze$. Hence, there is an
 element $a\in G$ such that $\rho(a)=\ze$. Let $G=\prod_{i=1}^sC_i$, where $C_i=\gnrsuch{a_i}{a_i^{o_i}=1}$ are cyclic groups. 
 We can suppose that $a_1=a$. Set $o=o_1$.
 Changing he generators $a_i$, we can make $\rho(a_i)=1$ for $i\ne1$. Let $G'=\gnr{a_2,a_3,\dots,a_s}$, so $G=C_1\xx G'$.
 Then $M\simeq M_1\*_\mZ \mZ$, where $M_1$ is $M$ considered as $C_1$-module and $\mZ$ is the trivial $G'$-module.
 Note that $M^G=0$, as $\ze v=v$ implies $v=0$. Hence $\hH^0(G,M)=0$.
 Consider the trace $T=\sum_{g\in G}g=(\sum_{k=0}^{o-1}a^k)(\sum_{g\in G'}g)$. Obviously, $\sum_{k=0}^{o-1}\ze^k=0$,
 hence $TM=0$. It implies that $\hH^{-1}(G,M)=H_0(G,M)=M/(\ze-1)M$. If $o=p^m$ fore some $m$, then also $o(\ze)=p^k$ for some $k$,
 whence $N_{\qK/\mQ}(1-\ze)=p$  \cite{bs} and $\hH^{-1}(G,M)=H_0(G,M)\simeq\mZ/p\mZ$. If $o(\ze)$ is not a degree of a prime number,
 then $N_{\qK/\mQ}(1-\ze)=1$ and $\hH^{-1}(G,M)=H_0(G,M)=0$ (it also follows from Corollary~\ref{34})..
 
 Let a finite abelian group $G$ be a direct product $G_1\xx G_2$ and the orders of $G_1$ and $G_2$ be coprime. 
 If $\qK_i\ (i=1,2)$ is a cyclotomic field arising from a simple $\mQ G_i$-module, then $\qK=\qK_1\*_\mQ\qK_2$ is 
 again a field, hence a simple $\mQ G$-module, and all simple $\mQ G$-modules arise in this way. If $M_i\ (i=1,2)$ is a
  $G_i$-lattice in $\qK_i$, then $M=M_1\*_\mZ M_2$ is a $G$-lattice in $\qK$, unique up to genus. 
  Corollary~\ref{34} shows that $\hH^n(G,M)=0$ if neither $M_1$ nor $M_2$ is trivial. If $M_1$ is non-trivial and $M_2$
  is trivial, then $\hH^n(G,M)\simeq\hH^n(G_1,M_1)$, and if both $M_1$ and $M_2$ are trivial, then
  $\hH^n(G,M)\simeq\hH^n(G_1,\mZ)\+\hH^n(G_2,\mZ)$. Thus we only need to consider the case of $p$-groups.
 Note also that 
  $\mT=\bop_p\mT_p$ and $\mT_p$ is the \emph{quasicyclic $p$-group}, i.e. the direct limit $\varinjlim_m\mZ/p^m\mZ$ 
  with respect to the natural embeddings $\mZ/p^m\,Z\to\mZ/p^{m+1}\mZ\,$. Hence, if $M$ is finitely generated, 
  $DM\simeq \bop_pDM_p$, where $D_pM=\Hom_\mZ(M,\mT_p)$. If $M$ is a lattice, the additive group of $D_pM$
  is a direct product of several copies of $\mT_p$. Moreover, if $G$ is a $p$-group, $\hH^n(G,D_qM)=0$ and
  $D_q\hH^n(G,M)=0$ for $q\ne p$, so we can always replace $D$ by $D_p$ in all formulae from Proposition~\ref{32}.
  
  So, let $G=\prod_{k=1}^sG_k$, where $G_k$ is a cyclic group of order $p^{m_k}$. We calculate cohomologies of a non-trivial
  irreducible $G$-lattices. Actually, it is easier to calculate homologies.
  
   \begin{theorem}\label{41} 
   Let $M$ be a non-trivial irreducible $G$-lattice. Then 
   \\
   $ H_n(G,M)\simeq(\mZ/p\mZ)^{\nu(n,s)}$,
   where
   \begin{equation}\label{e42} 
   \nu(n,s)=(-1)^n\sum_{i=0}^{n}\binom{-s}{i}.
   \end{equation}
  \end{theorem}

   Note that for fixed $n$ the value of $\nu(n,s)$ is a polynomial of degree $n$ with respect to $s$ with the leading coefficient $(n!)^{-1}$. 
   For instance,
   \begin{align*}
   \nu(0,s)&=1,\\
   \nu(1,s)&=s-1,\\
   \nu(2,s)&=\frac{s^2+s+2}2,\\
   \nu(3,s)&=\frac{s^3+5s-6}6.
   \end{align*} 
   
   \begin{proof}
     We consider $G$ as a direct product $G'\xx G_s$, where $G'=\prod_{i=1}^{s-1}G_i$, and suppose that $G_s$ acts trivially on $M$.
   Then $M$ can be considered as the outer tensor product $M'\xx_\mZ\mZ$, where $M'=M$ considered as $G'$-module and $\mZ$ is
   considered as trivial $G_s$-module. Then we can use the K\"unneth formula \cite[Corollary~V.5.8]{br}:
   \begin{multline}\label{e40} 
           H_n(G,M)
         \simeq\big(\bop_{i=0}^n H_i(G',M')\*_\mZ H_{n-i}(G_s,\mZ)\big) \+\\
   \+    \big( \bop_{i=0}^{n-1}\tor_1^\mZ(H_i(G',M'),H_{n-i-1}(G_s,\mZ))\big).
   \end{multline}
   Recall that, for a cyclic group $C=\mZ/p^m\mZ$,
   \begin{align*}
    H_0(C,\mZ)&=\mZ;\\
    H_n(C,\mZ)&=
    \begin{cases}
     \mZ/p^m\mZ &\text{if $n$ is odd,}\\
     0 &\text{if $n$ is even;}
    \end{cases}\\
     \intertext{while for a non-trivial irreducible lattice $M$}
     H_n(C,M)&=
     \begin{cases}
      \mZ/p\mZ&\text{if $n$ is even,}\\
      0 &\text{if $n$ is odd,}
     \end{cases}\\
     \intertext{that is}
     \nu(n,1)&=
     \begin{cases}
      1&\text{if $n$ is even,}\\
      0 &\text{if $n$ is odd.}
     \end{cases}\\
     \intertext{Moreover,}
     H_0(G,M)&=\mZ/p\mZ,\\
     \intertext{that is}
     \nu(0,s)&=1.
   \end{align*}
      Thus \eqref{41} is valid for $n=0$ and for $s=1$, the minimal values of $n$ and $s$.
   Therefore, the K\"unneth formula implies that $H^n(G,M)\simeq(\mZ/p\mZ)^{\nu(n,s)}$ for some $\nu(n,s)$. Moreover, it implies that
   \[
      \textstyle
    \nu(n,s)=\sum_{k=0}^n\nu(n,s-1)=\nu(n,s-1)+\nu(n-1,s) \\
   \]
   Hence we can prove \eqref{e41} by induction, supposing that it is true for $\nu(n,s-1)$ and $\nu(n-1,s)$. Then we have
   \begin{align*}
   \textstyle
   \nu(n,s)&=\nu(n,s-1)+\nu(n-1,s)= \\
   &=(-1)^n\sum_{i=0}^n\binom{-s+1}i -(-1)^n\sum_{i=0}^{n-1}\binom{-s}i = \\
   &=(-1)^n\sum_{i=0}^n\left(\binom{-s+1}i-\binom{-s}{i-1}\right) =\\
   &=(-1)^n\sum_{i=0}^n\binom{-s}i.
   \qedhere
   \end{align*}
   \end{proof}
      
   Note that in this case $\hH^{-1}(G,M)=H_0(G,M)$ and $\hH^0(G,M)=0$.
   
  The formulae \eqref{e41} and \eqref{e42} give the following result.
     
   \begin{corol}\label{42} 
   If $M$ is a non-trivial irreducible $G$-lattice, then
       \[
        \hH^n(G,M)\simeq\hH^{n-1}(G,DM)\simeq (\mZ/p\mZ)^{\nu(|n|-1,s)}.
       \]    
   \end{corol}
   
 Analogous calculations give the known result for the trivial $G$-module $\mZ$ (cf. \cite{lyn,tak}).

   \begin{theorem}\label{43} 
    If $n\ne0$ and $m_1\ge m_2\ge \dots \ge m_s$, then
   \begin{equation}\label{e43} 
     \textstyle
    \hH^n(G,\mZ)\simeq \bop_{k=1}^s (\mZ/p^{m_k}\mZ)^{\nu(|n|-1,k)+(-1)^n}.
    \end{equation}
   \end{theorem}
   \noindent
   Recall that $\hH^0(G,\mZ)\simeq\mZ/p^m\mZ$, where $m=\sum_{k=1}^sm_k$.
   \begin{proof}
    First of all, the K\"unneth formula \eqref{e40} implies that $H_n(G,\mZ)$ is a direct sum of $\mu(n,s)$ cyclic groups
    so that
   \begin{align*}
   &  \mu(n,s)=\sum_{i=1}^n\mu(i,s-1)+\eps,\\
   \intertext{where}
   & \eps=
   \begin{cases}
    1 &\text{if $n$ is odd},\\
    0 &\text{if $n$ is even},
   \end{cases}\\
   \intertext{whence}
   & \mu(n,s)=\mu(n,s-1)+\mu(n-1,s)-(-1)^{n-1}.\\
   \intertext{Using inductiion by $s$, we obtain that}
   & \mu(n,s)=\nu(n,s)+(-1)^n,\\
   \intertext{hence}
   & \mu(n,s)=\mu(n,s-1)+\nu(n-1,s).\\ 
   \intertext{Note that all groups $H^i(G_s,\mZ)$ are of period $p^{m_s}$. Therefore, by \eqref{e40},}
     &H_n(G,\mZ)\simeq H_n(G',\mZ)\+ (\mZ/p^{m_s}\mZ)^r\\
     \intertext{for some $r$. Together with the formula for $\mu(n,s)$, it gives that}
     &H_n(G,\mZ)\simeq H_n(G',\mZ)\+ (\mZ/p^{m_s}\mZ)^{\nu(n-1,s)-(-1)^n}.\\
     \intertext{By induction, we obtain that}
  & \textstyle     H_n(G,\mZ)
   \simeq  \bop_{k=1}^s (\mZ/p^{m_k}\mZ)^{\nu(n-1,k)-(-1)^n}.
  \end{align*}
   In view of \eqref{e41}, it is just the formula  \eqref{e43}.
   \end{proof}
 
  \section{Explicit formulae}
  \label{s5} 
  
   In this section we find explicit formulae for crossed homomoprhisms (elements of $H^1(G,M)$) and cocycles (elements of
   $H^2(G,M)$) for irreducible latticies and their duals (the latter are important, for instance, in study of Chernikov groups
   see \cite{gud}). We use the resolution defined in Section~\ref{s1}.
   
   Let $G=\prod_{i=1}^{m_s}G_i$, where $G_i=\gnrsuch{a_i}{a_i^{p^{m_i}}=1}$ is a cyclic group of order
  $o_i=p^{m_i}$. We set $s_i=s_{a_i}$. For a cochain $\mu:\mP_n\to M$ we denote by $\dd\mu$ its coboundary, that is
  the composition $\mu d_{n+1}:\mP_{n+1}\to M$. Then, if $\xi:\mP_1\to M$, $i<j$,
  \begin{equation}\label{e51} 
  \begin{split}
   &\dd\xi(x_i^2)=s_i\xi(x_i),\\
   &\dd\xi(x_ix_j)=(a_i-1)\xi(x_j)-(a_j-1)\xi(x_i).
  \end{split}
  \end{equation}
  Thus $\xi$ is a cocycle \iff 
  \begin{equation}\label{e52} 
   \begin{split}
    &s_i\xi(x_i)=0 \ \text{ for all } i,\\
    &(a_i-1)\xi(x_j)=(a_j-1)\xi(x_i) \ \text{ for all } i\ne j.
   \end{split}
    \end{equation}    
     If $\ga:\mP_2\to M$, $i<j<k$, then   
  \begin{align*}\label{e53} 
  \hhh & \dd\ga(x_i^3)=(a_i-1)\ga(x^2)=0,\\
  \hhh & \dd\ga((x_i^2x_j)= s_i \ga(x_ix_j)+(a_j-1)\ga(x_i^2),\\
  \hhh & \dd\ga(x_ix_j^2)= (a_i-1)\ga(x_j^2)-s_j\ga(x_ix_j),\\
  \hhh & \dd\ga(x_ix_jx_k)= (a_i-1)\ga(x_jx_k)-(a_j-1)\ga(x_ix_k)+(a_k-1)\ga(x_ix_j).
  \end{align*}
  Thus $\ga$ is a cocycle \iff  
    \begin{equation}\label{e54} 
   \begin{split}
    &(a_i-1)\ga(x^2_i)=0 \ \text{ for all } i,\\
    &s_i\ga(x_ix_j)=-(a_j-1)\ga(x_i^2),\\
    &s_j\ga(x_jx_i)=(a_i-1)\ga(x_j^2),\\
    &(a_j-1)\ga(x_ix_k)=(a_i-1)\ga(x_jx_k)+(a_k-1)\ga(x_ix_j).
   \end{split}
    \end{equation}
    
    Finally, if we identify an element $u\in M$ with the homomorphism $\mP_0\to M$ which maps $a$ to $au$, then 
    $\dd u(x_i)=(a_i-1)u$.
  
   First suppose that $M=\mZ$. 
   Then the element $s_i$ acts on $M$ as $p^{m_i}$ and the formulae \eqref{e52} show that $H^1(G,\mZ)=0$.       
   As $a_i-1$ acts as $0$, the formulae \eqref{e54} mean that $\ga$ is a cocycle \iff  $\ga(x_ix_j)=0$. 
   The formulae \eqref{e51} imply that, adding a coboundary, we can reduce $\ga(x_i^2)$ modulo $p^{m_i}$. Therefore, 
   $H^2(G,\mZ)\simeq\bop_{i=1}^{m_s}\mZ/p^{m_i}\mZ$ and generators of this group can be chosen as the cohomology
   classes of the cocycles $\ga_k:\mP_2\to\mZ$ such that $\ga_k(x_ix_j)=0$ for all $i,j$ and $\ga_k(x_i^2)=\delta_{ik}$.
   
   For the dual module $D_p\mZ=\mT_p$, the formulae \eqref{e52} mean that $\xi$ is a cocylce \iff $p^{m_i}\xi(x_i)=0$.
   Hence $ H^1(G,\mT_p)\simeq\bop_{i=1}^{m_s}\mT_{m_i}$, where $\mT_{m_i}=\setsuch{u\in\mT_p}{p^{m_i}u=0}$ 
   (it is a cyclic group of order $p^{m_i}$).
   As $\mT_p$ is divisible, the formulae \eqref{e51} imply that, adding a coboundary to a $2$-dimensional cocycle $\ga$, 
   one can always make $\ga(x_i^2)=0$. Then the formulae \eqref{e54} mean that $p^{m_{ij}}\ga_{x_ix_j}=0$, where 
   $m_{ij}=\min\{m_i,m_j\}$. Hence $ H^2(G,\mT_p)\simeq\bop_{i<j}\mT_{m_{ij}}\simeq\bop_{i<j}\mZ/p^{m_{ij}}\mZ$, and 
   generators of this group are the classes of cocycles $\ga_{kl}\ (1\le k<l\le s)$ such that $\ga_{kl}(x_i^2)=0$ for all $i$,
   while $\ga_{kl}(x_ix_j)=\de_{ki}\de_{lj}u_{kl}$, where $u_{kl}$ is a fixed element of $\mT_p$ of order $p^{m_{kl}}$.
   
  Let now $M$ be a lattice in a cyclotomic field $\qK$ of order $p^m$ such that $a_1$ acts as the multiplication by the  
  primitive root $\ze$ of unity of order $p^m$ and all $a_i\ (i>1)$ act trivially. As we can choose any
  lattice in the same genus, we can suppose that $M=\mZ[\ze]$. Therefore, 
  the formulae \eqref{e52} show that $\xi$ is a cocycle \iff $\xi(x_i)=0$ for $i>1$. As $\ze-1$ is a prime element in $\mZ[\ze]$ 
  with the norm $p$ \cite{bs}, $M/(\ze-1)M\simeq\mZ/p\mZ$. Hence, adding a coboundary $\dd u$ to $\xi$, one can make 
  $\xi(x_1)=\la$, where $\la\in\mZ$ is defined modulo $p$. Thus $ H^1(G,M)\simeq\mZ/p\mZ$. The formulae \eqref{e54} 
  show that $\ga$ is a cocycle \iff $\ga(x_1^2)=0$, $\ga(x_ix_j)=0$ if $1<i<j$ and $p^{m_i}\ga(x_1x_i)=(\ze-1)\ga(x_i^2)$. 
  The formulae \eqref{e51} imply that, adding a coboundary, one can make $\ga(x_1x_i)=\la_i$, where $\la_i\in\mZ$ is 
  defined modulo $p$. Then $\ga(x_i^2)$ is uniquely defined. Thus $ H^2(G,M)\simeq(\mZ/p\mZ)^{s-1}$. The generators
  of this group are the classes of cocycles $\ga_k\ (1<k\le s)$ such that $\ga_k(x_1^2)=\ga_(x_ix_j)=0$ for all $1<i<j$,
  $\ga_k(x_1x_i)=\de_{ik}$, $\ga_k(x_i)^2=0$ if $i\ne k$ and $(1-\ze)\ga_k(x_k)=p^{m_k}$. 
  
  Consider the dual module $D_pM$. As the multiplication by $\ze-1$ is injective on $M$, it is surjective on $D_pM$.
  On the other hand, the subgroup $\setsuch{u\in D_pM}{(\ze-1)u=0}$ is dual to $M/(\ze-1)M$,
  so it is generated by one element $u_0$ of period $p$. Thus, adding a couboundary $\dd u$ to a $1$-cocycle $\xi$, 
  one can make $\xi(x_1)=0$. Then $(\ze-1)\xi(x_i)=0$ if $i>1$, whence $\xi(x_i)=\la_iu_0$, where $\la_i\in\mZ/p\mZ$.  
  Hence $ H^1(G,D_pM)\simeq\qP_1^{s-1}\simeq(\mZ/p\mZ)^{s-1}$.  
  In the same way, adding a coboundary to a $2$-cocycle $\ga$, we can make $\ga(x_1x_i)=0$ for $i>1$. 
  Then the conditions \eqref{e54} give $(\ze-1)\ga(x_i^2)=0$ for all $i$, whence $\ga(x_i^2)=\la_iu_0$ ($\la_i\in\mZ/p\mZ$), 
  and $(\ze-1)\ga(x_ix_j)=0$ for $1<i<j$, whence $\ga(x_ix_j)=\la_{ij}u_0$ ($\la_{ij}\in\mZ/p\mZ$). 
  Therefore $ H^2(G,D_pM)\simeq\mT_1^{(s^2-s+2)/2}\simeq(\mZ/p\mZ)^{(s^2-s+2)/2}$. Th generators of this group
  are cocycles $\ga_k\ (1\le k\le s)$ and $\ga_{kl}\ (1<k<l\le s)$ such that $\ga_k(x_1x_i)=\ga_{kl}(x_1x_i)$ for $i>1$,
  $\ga_k(x_i^2)=\de_{ik}u_0$, $\ga_k(x_ix_j)=0$ for $i\ne j$, $\ga_{kl}(x_i^2=0$ for all $i$ and
  $\ga_{kl}(x_ix_j)=\de_{ik}\de_{jl}u_0$.


 \end{document}